\documentclass[12pt]{amsart} 
\usepackage{amssymb,amsmath,amsthm,amscd}
\usepackage[all]{xy}
\usepackage{amsfonts}
\usepackage{float}
\usepackage[shortlabels]{enumitem}

\usepackage{graphicx}

\usepackage{hyperref}

\numberwithin{equation}{section}

\newtheoremstyle{fancy1}{10pt}{10pt}{\itshape}{12pt}{\textsc\bgroup}{.\egroup}{8pt}{ }
\newtheoremstyle{fancy2}{10pt}{10pt}{}{12pt}{\itshape}{.}{8pt}{ }

\theoremstyle{fancy1}

\newtheorem{cor}[equation]{corollary}
\newtheorem{lem}[equation]{lemma}
\newtheorem{prop}[equation]{Proposition}
\newtheorem{theorem}[equation]{Theorem}

\newtheorem{main}{Theorem}

\newtheorem*{main*}{Theorem}
\newtheorem*{conj*}{Conjecture}
\newtheorem*{cor*}{corollary}

\newtheorem*{question*}{Question}

\newtheorem*{problem*}{Problem}

\theoremstyle{fancy2}
\newtheorem{definition}[equation]{Definition}
\newtheorem{rem}[equation]{Remark}
\newtheorem*{rem*}{Remark}

\newtheorem*{rems*}{Remarks}

\newcommand{\cref}[1]{Corollary~\ref{#1}}

\newcommand{\lref}[1]{Lemma~\ref{#1}}
\newcommand{\pref}[1]{Proposition~\ref{#1}}
\newcommand{\rref}[1]{Remark~\ref{#1}}
\newcommand{\tref}[1]{Theorem~\ref{#1}}
\newcommand{\sref}[1]{Section~\ref{#1}}

\newcommand{\e}{\epsilon}

\newcommand{\ggm}{geometric graph manifold\ }

\newcommand{\ggms}{geometric graph manifolds\ }

\newcommand{\Sph}{\mathbb{S}}

\newcommand{\su}{L^2}

\newcommand{\met}{\la\ ,\,\ra}

\newcommand{\M}{{\mathcal B}}

\newcommand{\Fol}{{\mathcal F}}

\newcommand{\gencyl}{(\su\times\R^{n-2})/G}
\newcommand{\gencylrec}{\su\times\R^{n-2}}

\newcommand{\cyl}{(\su\times\R^{n-2})/G}
\newcommand{\tcyl}{\su\times\R^{n-2}}

\newcommand{\s}{{\mathcal{S}}}
\newcommand{\C}{{\mathbb{C}}}
\newcommand{\R}{{\mathbb{R}}}
\newcommand{\Z}{{\mathbb{Z}}}
\newcommand{\Q}{{\mathbb{Q}}}
\newcommand{\N}{{\mathbb{N}}}
\newcommand{\g}{{{K}}}

\newcommand{\SO}{\ensuremath{\operatorname{SO}}}

\newcommand{\Sp}{\ensuremath{\operatorname{Sp}}}

\newcommand{\oo}{{\mathfrak{o}}}

\def\con#1=#2(#3){#1 \equiv #2 \bmod{#3}}

\newcommand{\ml}{\langle} 
\newcommand{\mr}{\rangle}
\newcommand{\la}{\langle}
\newcommand{\ra}{\rangle}

 \DeclareMathOperator{\Iso}{Iso}

\DeclareMathOperator{\SL}{SL}
\DeclareMathOperator{\Diff}{Diff}

\begin{document}

\title{Geometric Graph Manifolds\\with non-negative scalar curvature} 

\author{Luis A. Florit}
\address{IMPA: Est. Dona Castorina 110, 22460-320, Rio de Janeiro,
Brazil}
\email{luis@impa.br}
\author{Wolfgang Ziller}
\address{University of Pennsylvania: Philadelphia, PA 19104, USA}
\email{wziller@math.upenn.edu}
\thanks{The first author was supported by CNPq-Brazil,
and the second author by a grant from the National Science
Foundation, by IMPA, and CAPES-Brazil}

\begin{abstract}
We classify $n$-dimensional \ggms with nonnegative scalar curvature
by first showing that if $n>3$, the universal cover splits off a
codimension 3 Euclidean factor. We then proceed with the classification
of the 3-dimensional case, where the condition is equivalent to the
eigenvalues of the Ricci tensor being $(\lambda,\lambda, 0)$ with
$\lambda\ge 0$. In this case we prove that such a manifold is either a
lens space or a prism manifold with a very rigid metric. This allows us
to also classify the moduli space of such metrics: it has infinitely many
connected components for lens spaces, while it is connected for prism
manifolds.
\end{abstract}
\maketitle

A \ggm $M^n$ is a Riemannian manifold which is the union of twisted
cylinders $C^n=\gencyl$, where $G\subset\Iso(\gencylrec)$ acts properly
discontinuously and freely on the Riemannian product of a connected
surface $\su$ with the Euclidean space $\R^{n-2}$. In addition, the
boundary of each twisted cylinder is a union of compact totally geodesic
flat hypersurfaces, each of which is isometric to a boundary component of
another twisted cylinder. In its simplest form, as first discussed in
\cite{g}, they are the union of building blocks of the form $\su\times
S^1$, where $\su$ is a surface, not diffeomorphic to a disk or an
annulus, whose boundary is a union of closed geodesics. The building
blocks are glued along common boundary totally geodesic flat tori by
switching the role of the circles. Such graph manifolds have been studied
frequently in the context of manifolds with nonpositive sectional
curvature. In fact, they were the first examples of such metrics with
geometric rank one. Furthermore, in \cite{sc} it was shown that if a
complete 3-manifold with nonpositive sectional curvature has the
fundamental group of a graph manifold, then it is isometric to a
\ggm\!\!.

One of the most basic features of \ggms is that their curvature
tensor has  nullity space of dimension at least $n-2$ everywhere.
This property by itself already guarantees that each finite volume
connected component of the set of non-flat points is a twisted cylinder,
and under some further weak assumptions, the manifold is
isometric to a \ggm in the above sense; see \cite{fz2}.
See also \cite{BKV} and references therein for extensive literature
on manifolds with nullity equal to $n-2$.

In dimension $3$, the nullity condition is equivalent to saying that the
eigenvalues of the Ricci tensor are $(\lambda,\lambda, 0)$, or to the
assumption, called cvc(0), that every tangent vector is contained in a
flat plane; see \cite{sw}. Notice that this is in fact the only choice
for the eigenvalues of the Ricci tensor where the metric is allowed to be
locally reducible.

This nullity condition also arose in a different context. In \cite{fz1}
it was shown that a compact immersed submanifold $M^n\subset\R^{n+2}$
with nonnegative sectional curvature is either diffeomorphic to the
sphere $\Sph^n$, isometric to a product of two convex hypersurfaces
$\Sph^{k}\times\Sph^{n-k}\subset \R^{k+1}\times\R^{n-k+1}$, isometric to
$(\Sph^{n-1}\times\R)/\Z$, or diffeomorphic to a lens space
$\Sph^3/\Z_p\subset\R^5$. In the latter case it was shown that each
connected component of the set of nonflat points is a twisted cylinder.
The present paper arose out of an attempt to understand the intrinsic
geometry of such metrics. We thus want to classify all compact \ggms with
nonnegative sectional curvature, or equivalently, with nonnegative scalar
curvature. Notice that under this curvature assumption compactness is
equivalent to finite volume.

\medskip
We first show that their study can be reduced to dimension three.

\begin{main}\label{nonneg}
Let $M^n$, $n\geq 4$, be a compact \ggm with nonnegative scalar
curvature. Then, the universal cover $\tilde M^n$ of $M^n$ splits off an
$(n-3)$-dimensional Euclidean factor isometrically, i.e.,
$\tilde M^n=N^3\times\R^{n-3}$. Moreover, either $M^n$ is flat, or
$N^3=\Sph^2\times\R$ splits isometrically, or $N^3=\Sph^3$ with a \ggm
metric.
\end{main}
By the splitting theorem, the curvature condition by itself already
implies that $\tilde M^n$ is isometric to a product $Q^k\times\R^{n-k}$
with $Q^k$ compact and simply connected, but it is surprisingly delicate
to show that $k\leq 3$.

\smallskip

In dimension three, the simplest nontrivial example of a \ggm with
nonnegative scalar curvature is the usual description of $\Sph^3$ as the
union of two solid tori $D^2\times S^1$ endowed with a product metric,
see \mbox{Figure 1}. If this product metric is invariant under
$\SO(2)\times \SO(2)$, we can also take a quotient by the cyclic group
generated by $R_p\times R_p^q$ to obtain a \ggm metric on any lens space
$L(p,q)=\Sph^3/\Z_p$. Here $R_p\in \SO(2)$ denotes the rotation of angle
$2\pi/p$.

\begin{figure}[!ht]
\centering
\includegraphics[width=0.3\textwidth]{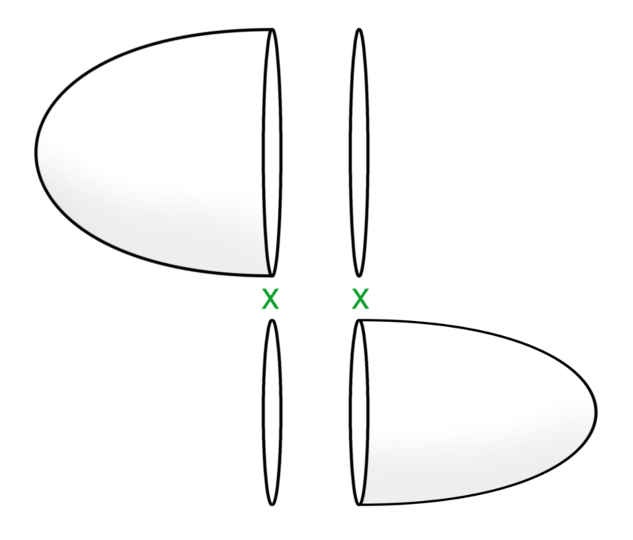}
\caption{\small $\Sph^3 \subset \R^5$ with
nonnegative curvature}
\end{figure}

There is a further family whose members also admit \ggm metrics with
nonnegative scalar curvature: the so-called {\it prism manifolds}
$P(m,n):=\Sph^3/G_{m,n}$, which depend on two relatively prime positive
integers $m,n$. Such a metric on $P(m,n)$ can be constructed as a
quotient of the metric on $\Sph^3$ as above by the group $G_{m,n}$
generated by $R_{2n} \times R_{2n}^{-1}$ and $(R_m \times R_m)\circ J$,
where $J$ is a fixed point free isometry switching the two isometric
solid tori. Topologically $P(m,n)$ is thus a single solid torus whose
boundary is identified to be a Klein bottle. Its fundamental group
$G_{m,n}$ is abelian if and only if $m=1$, and in fact $P(1,n)$ is
diffeomorphic to $L(4n,2n-1)$; see \sref{prelim}. Unlike in the case of
lens spaces, the diffeomorphism type of a prism manifold is determined by
its fundamental group.

\medskip

Our main purpose is to show that these are the only three dimensional
compact \ggms with nonnegative scalar curvature, and to classify the
moduli space of such metrics. We will see that the twisted cylinders in
this case are of the form $C=(D\times\R)/\Z$, where $D$ is the interior
of a 2-disk of nonnegative Gaussian curvature, whose boundary
$\partial D$ is a closed geodesic along which the curvature vanishes to
infinite order. We fix once and for all such a metric $\ml\,,\mr_0$ on a
2-disc $D_0$, whose boundary has length~$1$ and which is rotationally
symmetric. We call a \ggm metric on a \mbox{3-manifold} {\it standard} if
the generating disk $D$ of a twisted cylinder $C$ as above is isometric
to the interior of $D_0$ with metric $r^2\ml\,,\mr_0$ for some constant
$r>0$. Observe that the projection of $\partial D\times\{s\}$ for
$s\in\R$ is a parallel foliation by closed geodesics of the flat totally
geodesic 2-torus $(\partial D \times\R)/\Z$.

\smallskip

We provide the following classification:

\begin{main}\label{ls}
Let $M^3$ be a compact \ggm with nonnegative scalar
curvature and irreducible universal cover. Then $M^3$ is diffeomorphic to
a lens space or a prism manifold. Moreover, we have either:
\vspace{-1pt}
\begin{itemize}[leftmargin=0.8cm]
\item[$a)$] $M^3$ is a lens space and $M^3=C_1 \sqcup T^2 \sqcup C_2$,
i.e., $M^3$ is isometrically the union of two twisted cylinders
$C_i=(D_i\times\R)/\Z$ over disks $D_i$ glued together along their common
totally geodesic flat torus boundary $T^2$. Conversely, any flat torus
endowed with two parallel foliations by closed geodesics uniquely defines
a standard \ggm metric on a lens space;
\item[$b)$] $M^3$ is a prism manifold and $M^3=C \sqcup K^2$, i.e., $M^3$
is isometrically the closure of a single twisted cylinder
$C=(D\times\R)/\Z$ over a disk $D$, whose totally geodesic flat interior
boundary is isometric to a rectangular torus $T^2$, and $K^2=T^2/\Z_2$ is
a Klein bottle. Conversely, any rectangular flat torus endowed with a
parallel foliation by closed geodesics uniquely defines a standard \ggm
metric on a prism manifold.
\end{itemize}
\vspace{-1pt}
In addition, any \ggm metric with nonnegative scalar curvature on $M^3$ is
isotopic, through \ggm metrics with nonnegative scalar curvature,
to a standard one.
\end{main}
We call $T^2$, respectively $K^2$, the {\it core} of the \ggm
and will see that it is in fact an isometry invariant.

Observe that a twisted cylinder with generating surface a disc is
diffeomorphic to a solid torus. In topology one constructs a lens space
by gluing two such solid tori along their boundary by an element of
$GL(2,\Z)$. In order to make this gluing into an isometry, we twist the
local product structure. An alternate way to view this construction is as
follows. Start with an arbitrary twisted cylinder $C_1$ and regard the
flat boundary torus as the quotient of $\R^2$ with respect to a lattice.
We can then choose a second twisted cylinder $C_2$ whose boundary is a
different fundamental domain of the same lattice, and hence the two
twisted cylinders can be glued with an isometry of the boundary tori.
We note that in principle, a twisted cylinder can also be flat, but
we will see that in that case it can be absorbed by one of the nonflat
twisted cylinders.

\smallskip

The diffeomorphism type of $M^3$ in \tref{ls} is determined by the
(algebraic) oriented slope between the parallel foliations of $T^2$ by
closed geodesics. As we will see, this is also an isometry invariant
$\s(M^3,\oo)\in\Q$ of $M^3$ which we call its {\it slope}, once
orientations $\oo$ of $M^3$ and its core are chosen; see \sref{lenss} for
the precise definition.

\begin{main}\label{diffeo}
Let $M^3$ be a compact \ggm of nonnegative scalar curvature with
irreducible universal cover and slope $\s(M^3,\oo)=q/p\in\Q$. Then, in
case~$(a)$ of \tref{ls}, $M^3$ is diffeomorphic to the lens space
$L(p,q)$, while in case $(b)$ it is diffeomorphic to the prism manifold
$P(q,p)$.
\end{main}

This result can be used to classify the moduli space of \ggm metrics.
We first deform any such metric in \tref{ls} to be standard,
preserving the metric on the torus $T^2$, and then deform $T^2$ to be the
unit square $S^1\times S^1$, while preserving also the sign of the scalar
curvature in the process. In case $(a)$, we can also make one of the
foliations equal to $S^1\times \{w\}$. The metric is then determined by
the remaining parallel foliation of the unit square by closed geodesics.
Since the
diffeomorphism type of a lens space $L(p,q)$ is determined by
$\pm q^{\pm 1}\!\!\mod p$, we conclude:

\begin{cor*}
The moduli space of \ggm metrics with nonnegative scalar curvature on a
lens space has infinitely many connected components, whereas on
a prism manifold $P(q,p)$ with $q>1$ it is connected.
\end{cor*}

We will see that the moduli space for the lens space $L(4p,2p-1)$ has a
special component arising from the fact that it is diffeomorphic to
$P(1,p)$.

\medskip

Finally, we apply our results, combined with those in \cite{fz2}, to the
class of compact \mbox{3-dimensional} manifolds $M^3$ with Ricci
eigenvalues $(\lambda,\lambda,0)$ for $\lambda\geq 0$. Theorem A in
\cite{fz2} implies that any connected component of the set $M'$ of
non-flat points of $M^3$ is isometric to a twisted cylinder. The basic
geometric feature of $M'$ is that it admits a parallel foliation by
complete geodesics tangent to the kernel of the Ricci tensor. If there
exists a larger open set $M''\supset M'$ which admits a parallel
foliation by complete geodesics extending that of $M'$, then any
connected component of $M''$ is still isometric to a twisted cylinder.
Such an extension $M''$ is called {\it full} if it is dense in $M^3$ and
if its collection of twisted cylinders is locally finite. From the
second theorem in \cite{fz2} we thus conclude the following.

\begin{cor*} Let $M^3$ be a compact Riemannian manifold with Ricci
eigenvalues $(\lambda,\lambda,0)$ for some function $\lambda\ge 0$. Then
$M^3$ is isometric to one of the manifolds in \tref{ls} if and only if
its set of nonflat points admits a full extension.
\end{cor*}
This applies of course if $M'$ is already dense, as long as it
satisfies the mild regularity assumption that its collection of twisted
cylinders is locally finite. Although in \cite{fz2} we built an explicit
example where $M'$ admits no full extension, we conjecture that it always
admits a full extension when $\lambda\geq 0$.

\medskip

The paper is organized as follows. In Section \sref{prelim} we recall
some facts about geometric graph manifolds. In \sref{nonnegn} we prove
\tref{nonneg} by showing that the manifold is a union of one or two
twisted cylinders over disks, while in \sref{lenss} we classify their
metrics.

\section{Preliminaries}\label{prelim}

Let us begin with the definition of twisted cylinders and \ggms\!\!.

\smallskip

Consider the cylinder $\su\times\R^{n-2}$ with its natural product
metric, where $\su$ is a
connected surface.
We call the  quotient
$$
C^n=\gencyl
$$
a {\it twisted cylinder}, where $G\subset\Iso(\gencylrec)$ acts properly
discontinuously and freely on $\gencylrec$, and $\su$ the {\it generating
surface} of $C^n$. We also say that $C^n$ is a twisted cylinder {\it
over} $\su$. The Euclidean factor induces a foliation $\Gamma$ on $C^n$
whose leaves will be called the {\it nullity leaves} of $C^n$. These
leaves are complete flat totally geodesic and locally parallel of
codimension~$2$. Such twisted cylinders are the building blocks of
geometric graph manifolds:

\medskip

{\it Definition.}
A complete connected Riemannian manifold $M^n$, $n\geq 3$, is called a
{\it \ggm} if $M^n$ is a locally finite disjoint union of twisted
cylinders $C_i$ glued together along disjoint compact connected
totally geodesic flat hypersurfaces $H_\lambda$ of $M^n$. That is,
$$
M^n\setminus W= \bigsqcup_\lambda H_\lambda,
\ \ \ {\rm where}\ \ \ W:=\bigsqcup_i C_i.
$$
See Figure~2 for a typical (4-dimensional) example, where each twisted
cylinder is just the isometric product $L^2\times S^1\times S^1$ of a
surface $L^2$ and a flat torus.

\begin{figure}[!ht]
\centering
\includegraphics[width=0.74\textwidth]{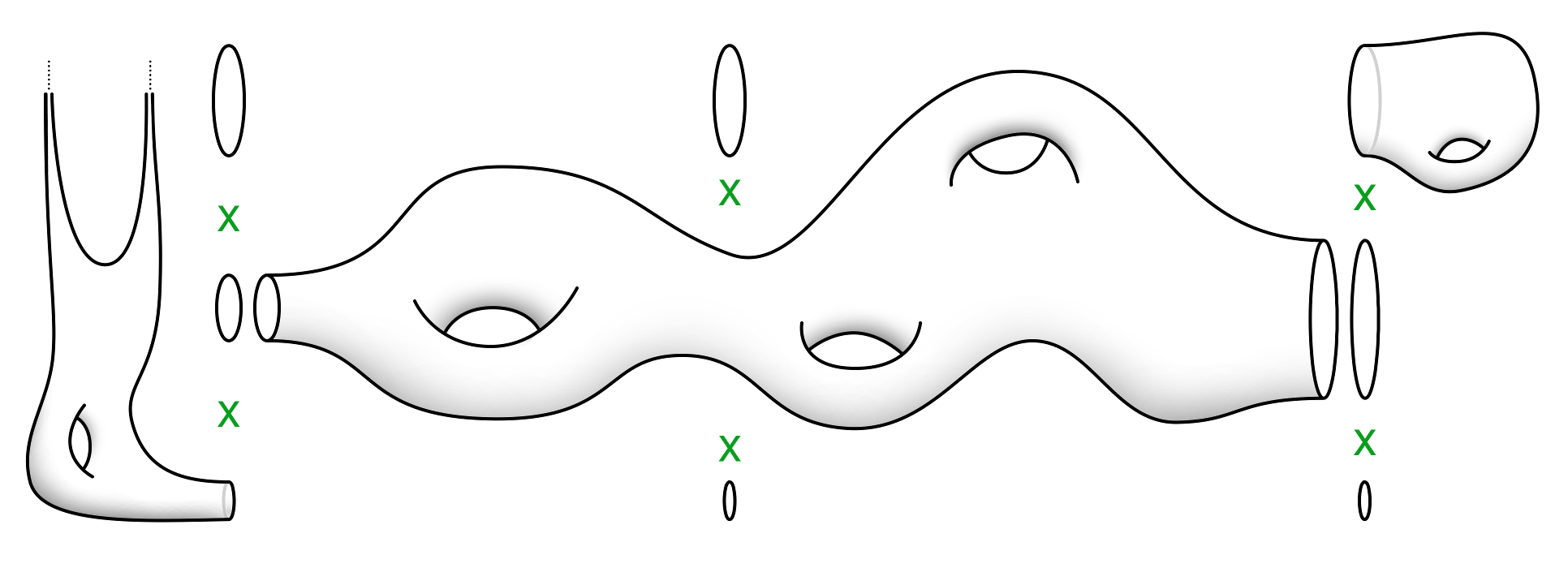}
\caption{\small An irreducible 4-dimensional \ggm}
\centerline{\small with three cylinders and two (finite volume) ends}
\end{figure}

\smallskip

We first make some general remarks about this definition.

\smallskip

\begin{enumerate}[labelindent=0pt,labelwidth=1em,label=\arabic*.,itemindent=0em,leftmargin=!]
\item We allow the possibility that the hypersurfaces $H_\lambda$ are
one-sided, even when $M^n$ is orientable.
\item The locally finiteness condition is equivalent to the assumption
that each $H_\lambda$ is a common boundary component of two twisted
cylinders $C_i$ and $C_j$, that may even be globally the same. When
$H_\lambda$ is one-sided it is a boundary component of only one twisted
cylinder.
\item As shown in \cite{fz2}, the foliations $\Gamma_i$ and $\Gamma_j$ of
$C_i$ and $C_j$ induce two totally geodesic foliations on $H_\lambda$.
When they agree, $C_i$, $C_j$ and $H_\lambda$ can be considered as a
single twisted cylinder. Thus, without loss of generality,
{\it we assume from now on that they are different}.
This implies that the generating surface $\su$
of each twisted cylinder~$C$ is the interior of a surface
with boundary consisting of complete geodesics along which the Gaussian
curvature vanishes to infinite order. We refer to these geodesics as
boundary geodesics of $\su$ itself.
\item These boundary geodesics of $\su$ do not have to be closed,
even when $C$ is compact.
\item The complement of $W$ is contained in the set of flat points of
$M^n$, but we do not require that the generating surfaces have
nonvanishing Gaussian curvature.
\item In principle, we could ask for the hypersurfaces $H_\lambda$ to
be complete instead of compact. However, compactness follows when $M^n$
has finite volume; see \cite{fz2}.
\item If none of the generating surfaces in a geometric graph manifold
are discs, it also admits a metric with nonpositive sectional curvature.
On the other hand, if all of the generating surfaces are discs, we will
see that it admits a metric with nonnegative sectional curvature.
\end{enumerate}

\medskip

In \cite{fz2} we gave a characterization of \ggms with finite volume in
terms of the nullity of the curvature tensor. But since a complete
noncompact manifold with nonnegative Ricci curvature has linear volume
growth by \cite{ya}, we will assume from now on that $M^n$ is compact.

\medskip

We now recall some properties of three dimensional lens spaces and prism
manifolds that will be needed later on; see e.g. \cite{st,hk,ru,or} for
details.

\medskip

One way of defining a lens space is as the quotient $L(p,q)=\Sph^3/\Z_p$,
where $g\in\Z_p\subset S^1\subset \C$ acts as $g\!\cdot\!(z,w)=(gz,g^qw)$
for $(z,w)\in\Sph^3\subset\R^4=\C^2$ for coprime integers $p,q$ with
$p\neq 0$. It is a well known fact that two lens spaces $L(p,q)$ and
$L(p,q')$ are diffeomorphic if and only if
\mbox{$q'=\pm q^{\pm 1}\!\!\mod p$}. An alternative description we will
use is as the union of two solid tori $D_i\times S^1$, with boundary
identified such that
$\partial D_1\times \{p_0\}\in \pi_1(\partial D_1\times S^1)$ is taken
into $(q,p)\in\Z\oplus\Z=\pi_1(\partial D_2\times S^1)$
with respect to its natural basis.

A prism manifold can also be described in two different ways.
The first one is to define it as the quotient
$\Sph^3/(H_1\times H_2)=H_1\backslash\Sph^3/H_2$, where
$H_1\subset\Sp(1)$ is a cyclic group acting as left translations on
$\Sph^3\simeq\Sp(1)$ and $H_2\subset\Sp(1)$ a binary dihedral group
acting as right translations. A more useful description for our purposes
is as the union of a solid torus $C=D\times S^1$ with the 3-manifold
\begin{equation}\label{pris}
N^3=(S^1\times S^1\times I)/\la(j,-Id)\ra, \ \ \ \ {\rm where}
\ \ \ j(z,w)=(-z,\bar w).
\end{equation}
Notice that $N^3$ is a bundle over the Klein bottle $K=T^2/\la j\ra$
with fiber an interval $I=[-\e,\e]$ and orientable total space.
Thus $\partial N^3$ is the torus $S^1\times S^1$, and we
glue the two boundaries via a diffeomorphism. Here
$\pi_1(N^3)=\pi_1(K)=\{ a,b\mid bab^{-1}=a^{-1}\}$ and
$\pi_1(\partial N^3)=\Z\oplus\Z$, with generators $a,b^2$, where $a$
represents the first circle and $b^2$ the second one. Then
$P(m,n)$ is defined as gluing $\partial C$ to $\partial N^3$ by sending
$\partial D\times \{p_0\}$ to $a^mb^{2n}\in\pi_1(\partial N^3)$. We can
again assume that $m,n>0$ with $\gcd(m,n)=1$. Furthermore,
$$
\pi_1(P(m,n))=G_{m,n}=\{a,b\mid bab^{-1}=a^{-1},\ a^mb^{2n}=1 \}.
$$
This group has order $4mn$ and its abelianization has order $4n$. Thus
the fundamental group determines and is determined by the ordered pair
$(m,n)$. In addition, $G_{m,n}$ is abelian if and only if $m=1$ in which
case $P(m,n)$ is diffeomorphic to the lens space $L(4n,2n-1)$. Unlike in
the case of lens spaces, the diffeomorphism type of $P(m,n)$ is uniquely
determined by $(m,n)$.
Prism manifolds can also be characterized as the 3-dimensional spherical
space forms which contain a Klein bottle, which for $m>1$ is also
incompressible. Observe in addition that in $N^3$ we can shrink the
length of the interval $I$ in \eqref{pris} down to 0, and hence $P(m,n)$
can also be viewed as a single solid torus whose rectangular flat torus
boundary has been identified to a Klein bottle, as in part $(b)$ of
\tref{ls}.

\section{A dichotomy and the proof of \tref{nonneg}}\label{nonnegn}

In this section we provide the general structure of \ggms with
nonnegative scalar curvature by showing a dichotomy: they are built from
either one or two twisted cylinders over 2-disks. This will then be used
to prove \tref{nonneg}.

\vspace{10pt}

Let $M^n$ be a compact nonflat \ggm with nonnegative scalar curvature. We
will furthermore assume that $M^n$ is not itself a twisted cylinder since
in this case the universal cover of $M^n$ is isometric to
$\Sph^2\times\R^{n-2}$, where $\Sph^2$ is endowed with a metric of
nonnegative Gaussian curvature. Recall that we also assume that the
nullity foliations of two twisted cylinders glued along a hypersurface
$H$ induce two different foliations on $H$, which in turn implies that
the Gaussian curvature of the two generating surfaces vanish to infinite
order along their boundary geodesic.

\smallskip

By assumption, there exists a collection of compact flat totally geodesic
hypersurfaces in $M^n$ whose complement is a disjoint union of (open)
twisted cylinders $C_i$. Let $C=\cyl$ be one of these cylinders whose
boundary in $M^n$ is a disjoint union of compact flat totally geodesic
hypersurfaces. There is also an {\it interior boundary} $\partial_i C$
of $C$, which we also denote for convenience as $\partial C$ by abuse of
notation. This boundary can be defined
as the set of equivalence classes of Cauchy sequences $\{p_n\}\subset C$
in the interior distance function $d_C$ of $C$, where
$\{p_n\}\sim\{p'_n\}$ if $\lim_{n\to\infty}d_C(p_n,p'_n)=0$. Since $M^n$
is compact, such a Cauchy sequence $\{p_n\}$ converges in $M^n$, and we
have a natural map $\sigma:\partial C\to M$ that sends $[\{p_n\}]$ to
$\lim_{n\to\infty}p_n\in M^n$. This map is, on each component of
$\partial C$, either an isometry or a locally isometric two-fold covering
map since $H=\sigma(\partial C)$ consists of disjoint smooth
hypersurfaces which are two-sided in the former case, and one-sided in
the latter. Therefore, $\partial C$ is smooth as well and
$C\sqcup\partial C$ is a closed twisted cylinder with totally geodesic
flat compact interior boundary, that by abuse of notation we still denote
by $C$. Similarly, $\su$ is a smooth surface with geodesic interior
boundary components along which the Gaussian curvature vanishes to
infinite order.

\vspace{10pt}

We first determine the generating surfaces of the twisted cylinders:

\begin{prop}\label{cyl} Let $C=\cyl$ be a compact twisted cylinder with
nonnegative curvature as above. Then one of the following holds:
\begin{itemize}
\item[i)]
The surface $\su$ is isometric to a 2-disk $D$ with nonnegative Gaussian
curvature, whose boundary is a closed geodesic along which the
curvature of $D$ vanishes to infinite order.
\item[ii)] $C$ is flat and there exists a compact flat hypersurface $S$
such that $C$ is isometric to either $[-{s_0},{s_0}]\times S$, or to
$([-{s_0},{s_0}]\times S)/\{(s,x)\sim (-s,\tau(x))\}$ for some involution
$\tau$ of $S$.
\end{itemize}
\end{prop}
\proof
Since $C$ is compact and the boundary is totally geodesic, we can apply
the soul theorem to $C$, see \cite{cg} Theorem 1.9 and \cite{pt} Theorem
4.1. Thus there exists a compact totally geodesic submanifold $S\subset
C$ and $C$ is diffeomorphic to the disc bundle
$D_\e(S)=\{v\in T_p C\mid v\perp T_pS,\ |v|\le\e\}$ for some $\e>0$.
Recall that $S$ is constructed as follows. Let
$C^s=\{p\in C\mid d(p,\partial C)\ge s\}$. Then $C^s$ is convex, and the
set of points $C^{s_0}$ at maximal distance $s_0$ from $\partial C$
is a totally geodesic submanifold, possibly with boundary. Repeating the
process if necessary, one obtains the soul $S$. In our situation, let
$q=[(p,v)]\in C^{s_0}$, and $\gamma$ a minimal geodesic from $q$ to
$\partial C$. Since it meets
$\partial C=((\partial \su)\times\R^{n-2})/G$ perpendicularly, we have
$\gamma=[(\alpha,v)]$ where $\alpha$ is a geodesic in the leaf
$\su_v=[\su\times\{v\}]$ meeting $\partial \su_v$ perpendicularly. So, for
every $w\in\R^{n-2}$, the geodesic $[(\alpha,w)]$ is also minimizing,
$[(p,w)]\in C^{s_0}$ lies at maximal distance $s_0$ to $\partial C$, and
hence $C^{s_0}=(T\times\R^{n-2})/G$ where $T\subset \su$ is a segment, a
complete geodesic or a point. Therefore $S=(T'\times\R^{n-2})/G$, where
$T'$ is a point or a complete geodesic (possibly closed).

We first consider the case where  $T'$ is a point and hence the soul is a single nullity leaf. Recall, that in order to show that $C$ is
diffeomorphic to the disc bundle
$D_\e(S)$, one constructs a gradient like vector field $X$ by observing that
the distance function to the soul has no critical points. In our case,
the initial vector to all minimal geodesics from $[(p,v)]\in C$ to~$S$
lies in the leaf $\su_v$ and hence we can construct $X$ such that $X$ is
tangent to $\su_v$ for all $v$. The diffeomorphism between $C$ and
$D_\e(S)$ is obtained via the flow of $X$, which now preserves the leaves
$\su_v$ and therefore $\su$ is diffeomorphic to a disc.

If $T'$ is a complete geodesic, the soul $S$ is flat and has
codimension 1. If $X$ is a unit vector field in $\su$ along $T'$ and
orthogonal to $T'$, it is necessarily parallel and its image under the
normal exponential map of $S$ determines a flat surface by Perelman's
solution to the soul conjecture, see \cite{pe}. This surface lies in
$\su$, and every point $q \in\su$ is contained in such a surface since we
can connect $q$ to $S$ by a minimal geodesic, which is contained in some
$L_v$, and is orthogonal to $T'$. Thus $\su$ is flat and hence either
$\su=T'\times [-{s_0},{s_0}]$, and hence $C=[-{s_0},{s_0}]\times S$, or
$\su$ is a Moebius strip and hence
$C=([-{s_0},{s_0}]\times S)/\{(s,x)\sim (-s,\tau(x))\}$ for some
involution $\tau$ of $S$.
\qed

\begin{rem}\label{noflat}
{\rm A flat twisted cylinder as in $(ii)$ can be absorbed by any cylinder
$C'$ attached to one of its boundary components by either attaching
$[-{s_0},{s_0}]$ to the generating surface of $C'$ in the first case, or
attaching $(0,s_0]$ in the second, in which case $\{0\}\times
(S/\tau)=S/\tau$ becomes a one sided boundary component of $C'$. {\it We will therefore
 assume from now on that the generating surfaces of all twisted
cylinders are 2-discs}.}
\end{rem}

\begin{rem}\label{ext}
{\rm The properties at the boundary $\gamma$ of a disk $D$ as in
\pref{cyl} are easily seen to be equivalent to the fact that the natural
gluing $D\sqcup(\gamma\times(-\e,0])$, $\gamma\cong\gamma\times\{0\}$, is
smooth when we consider on $\gamma\times (-\e,0]$ the flat product metric.
In fact, in Fermi coordinates $(s\geq 0,t)$ along $\gamma$, the metric is
given by $ds^2 +f(t,s)dt^2$. The fact that $\gamma$ is a (unparameterized)
geodesic is equivalent to $\partial_sf(0,t)=0$, while the curvature
condition is equivalent to $\partial^k_sf(0,t)=0$ for all $t$ and
$k\ge 2$. Therefore, $f(s,t)$ can be extended smoothly as $f(0,t)$ for
$-\e<s<0$, which gives the smooth isometric attachment of the flat
cylinder $\gamma\times(-\e,0]$ to $D$.}
\end{rem}

As a consequence of \pref{cyl}, and the assumption that there are no flat cylinders, $\partial C=(\gamma\times \R^{n-2})/G$ is
connected, and so is $H=\sigma(\partial C)$. In particular, $M^n$
contains at most two twisted cylinders with nonnegative curvature glued
along $H$. We call such a connected compact flat totally geodesic
hypersurface $H$ a {\it core} of $M^n$. We conclude:

\begin{cor}\label{twopos} If $M^n$ is not flat and not itself a twisted
cylinder, then $M^n=W\sqcup H$ with core $H$, and either:
\vspace{-5pt}
\begin{enumerate}
\item[$a)$] $H$ is two-sided, $\sigma$ is an isometry, and
$W=C\sqcup C'$ is the disjoint union of two open nonflat twisted
cylinders as above attached via an isometry
$\partial C\simeq H\simeq \partial C'$; or
\item[$b)$] $H$ is one-sided, $\sigma$ is a locally isometric two-fold
covering map, $W=C$ is a single open nonflat twisted cylinder as above,
and $M^n=C \sqcup H = C \sqcup (\partial C/\Z_2)$.
\end{enumerate}
Furthermore, in case $(a)$, if $H'\subset M^n$ is an embedded compact
flat totally geodesic hypersurface then there exists an isometric product
$H\times[0,a]\subset M^n$, with $H=H\times \{0\}$ and $H'=H\times \{a\}$.
In particular, any such $H'$ is a core of $M^n$, and hence
the core is unique up to isometry. On the other hand, in case $(b)$ the
core $H$ is already unique.
\end{cor}
\proof
We only need to prove the uniqueness of the cores. In order to do
this, any limit of nullity leaves of $C$ at its boundary in $M^n$
will be called a boundary nullity leaf, or BNL for short.

For case $(a)$, first assume that $H\cap H'\neq\emptyset$ and take
$p\in H\cap H'$. Then a BNL of $C$ in $H$ at $p$ is contained in $H'$.
Indeed if not, the product structure of the universal cover
$\pi:\tilde C=\su\times\R^{n-2}\to C$, together with the fact that $H'$
is flat totally geodesic and complete and intersects $H$ transversely,
would imply that $L^2$, and hence $C$, is flat since by dimension reasons
the projection of $\pi^{-1}(H'\cap C)$ onto $L^2$ would be a surjective
submersion. Analogously, the (distinct) BNL of $C'$ at $p$ lies in $H'$,
and since $H$ is the unique hypersurface containing both BNL's, we have
that $H=H'$. If, on the other hand, $H\cap H'=\emptyset$, we can assume
$H'\subset C=(L^2\times\R^{n-2})/G$. Again by the product structure of
$\tilde C$ and the fact that $H'$ is embedded we see that
$H'=(\gamma'\times\R^{n-2})/G'$ where $\gamma'\subset L^2$ is a simple
closed geodesic and $G'\subset G$ the subgroup preserving $\gamma'$.
Since the boundary $\gamma$ of $L^2$ is also a closed geodesic and $L^2$
is a 2-disk with nonnegative Gaussian curvature, by Gauss--Bonnet there
is a closed interval $I=[0,a]\subset \R$ such that the flat strip
$\gamma\times I$ is contained in $L^2$, with $\gamma=\gamma\times\{0\}$
and $\gamma'=\gamma\times \{a\}$. Thus $G'$ acts trivially on $I$, which
implies our claim.

In case $(b)$ we have that $H\cap H'=\emptyset$ as in case $(a)$ since at
any point $p\in H$ we have two different BNL's at $\sigma^{-1}(p)$. Hence
as before $H'=(\gamma'\times\R^{n-2})/G'\subset C$ and
$H\times[0,a]\subset M^n$, with $H=H\times \{0\}$ and $H'=H\times \{a\}$.
But then the normal bundle of $H'$ is trivial, contradicting the fact
that $H$ is one-sided.
\qed

\begin{rem}\label{2fold}
{\rm Any manifold in case $(b)$ admits a two-fold cover whose covering
metric is as in case $(a)$. Indeed, we can attach to~$C$ another copy of
$C$ along its interior boundary $\partial_i C$ using the involution that
generates $\Z_2$. Switching the two cylinders induces the two-fold cover
of $M^n$.}
\end{rem}

\vspace{6pt}

We proceed by showing that our \ggms are essentially 3-dimensional.
Observe that we only use here that $M^n\setminus W$ is connected, with no
curvature assumptions. In fact, the same proof shows that if
$M^n\setminus W$ has $k$ connected components, then $M^n$ splits off an
$(n-k-2)$-dimensional Euclidean factor.

\medskip

{\it Claim. If $n>3$, the universal cover of $M^n$ splits off an
$(n-3)$-dimensional Euclidean factor.}
\proof
Assume first that $M^n$ is the union of two cylinders $C$ and $C'$ with
common boundary~$H$. Consider the nullity distributions $\Gamma$ and
$\Gamma'$ on the interior of $C$ and $C'$, which extend
uniquely to parallel codimension one distributions $F$ and $F'$ on
$H$, respectively. Recall that $F\ne F'$ since otherwise the
universal cover is an isometric product $N^2\times\R^{n-2}$. So
$J:=F\cap F'$ is a codimension two parallel distribution on $H$.
We claim that $J$ extends to a parallel distribution on the interior of
both $C$ and $C'$.

To see this, we only need to argue for $C$, so lift the
distributions $J$ and $F$ to the cover $S^1\times \R^{n-2}$ of $H$
under the projection $\pi\colon \tcyl\to C=\cyl$, and
denote these lifts by $\hat J$ and $\hat F$. They are again
parallel distributions whose leaves project to those of $J$ and $F$
under $\pi$. At a point $(x_0,v_0)\in S^1\times \R^{n-2}$ a leaf of
$\hat F$ is given by $\{x_0\}\times \R^{n-2}$ and hence a leaf of
$\hat J$ by $\{x_0\}\times W$ for some affine hyperplane
$W\subset\R^{n-2}$. Since $\hat J$ is parallel, any other leaf is given
by $\{x\}\times W$ for $x\in S^1$. Since $G$ permutes the leaves of
$\hat F$, $W$ is invariant under the projection of $G$ into
$\Iso(\R^{n-2})$. Therefore $\pi(\{x\}\times W)$ for $x\in\su$ are the
leaves of a parallel distribution on the interior of $C$, restricting to
$J$ on its boundary.

Therefore, we have a global flat parallel distribution $J$ of codimension
three on $M^n$, which implies that the universal cover splits
isometrically as $N^3\times \R^{n-3}$.

Now, if $M^n$ consists of only one open cylinder $C$ and its one-sided
boundary, by \rref{2fold} there is a two-fold cover $\hat M^n$
of $M^n$ which is the union of two cylinders as above and whose universal
cover splits an $(n-3)$-dimensional Euclidean factor.
\qed
\vspace{1.5ex}

We can now finish the proof of \tref{nonneg}. Since $M^n$ is compact with
nonnegative curvature, the splitting theorem implies that the universal
cover splits isometrically as $\tilde M^n=Q^k\times \R^{n-k}$ with $Q^k$
compact and simply connected. According to the above claim, $k=2$ and
hence $Q^2\simeq\Sph^2$, or $k=3$ and by Theorem 1.2 in \cite{h} we have
$Q^3\simeq \Sph^3$. In the latter case, we claim that the metric on
$\Sph^3$ is again a \ggm metric. Indeed, if $\sigma\colon \Sph^3\times
\R^{n-3} \to M^n$ is the covering map, and $C\subset M^n$ a twisted
cylinder, then in $C'=\sigma^{-1}(C)$ the codimension $2$ nullity leaves
contain the $\R^{n-3}$ factor. Since the universal cover of $C'$
has the form $\su\times \R^{n-2}$, the metric on $\Sph^3$ must be a
\ggm metric.

\smallskip

\section{Geometric graph 3-manifolds with nonnegative curvature}\label{lenss}

\smallskip

In this section we classify 3-dimensional \ggms with nonnegative scalar
curvature, giving an explicit construction of all of them. As a
consequence, we show that, for each lens space, the number of connected
components of the moduli space of such metrics is infinite, while for
each prism manifold, the moduli space is connected. Recall that we assume
that $M^3$ itself is not a single twisted cylinder. Furthermore, none of
the twisted cylinders are flat, hence their generating surfaces are discs
and $M^3$ is the union of one or two twisted cylinders according to the
dichotomy in \cref{twopos}.

\medskip

Let $M^3$ be such a compact \ggm with nonnegative scalar curvature. We
first observe that $M^3$ is orientable. Indeed, by \tref{nonneg},
$M^3=\Sph^3/\Pi$ for some finite group $\Pi$ acting freely.
Moreover, if an element $g\in\Pi$ reverses orientation,
the Lefschetz fixed point theorem implies that $g$ has a fixed point.
Thus every cylinder $C=(D\times \R)/G$ is orientable as well, i.e. the
action of $G$ preserves orientation.

\smallskip

For $g\in G$, we write $g=(g_1,g_2)\in\Iso(D\times\R)$. Thus $g_1$
preserves the closed geodesic $\partial D$ and fixes the soul point
$x_0\in D$. If $g\neq e$ and $g_1$ reverses orientation, then so does
$g_2$ and hence~$g$ would have a fixed point. Thus $g_2$ preserves
orientation and is a translation which is nontrivial since $g_1$ has a
fixed point. This easily implies that $G= \Z$. Altogether, the twisted
cylinders are of the form $C=(D\times\R)/\Z$ with $\Z$ generated by some
$g=(g_1,g_2)$. If $g_1$ is nontrivial, then $g_1$ is determined by its
derivative at $x_0$. After orienting $D$, $d(g_1)_{x_0}$ is a rotation
$R_{\theta}$ of angle $2\pi\theta$, $0\leq \theta<1$. We simply say that
$g_1$ acts as a rotation $R_{\theta}$ on~$D$. Thus $g$ acts via
\begin{equation}\label{g} g(x,s)=(R_{\theta}(x),s+h)\in\Iso(D\times\R),
\end{equation} for a certain constant $h>0$, after orienting the nullity
distribution $\Gamma\cong T^\perp D$. We can regard $\theta$ as the twist
of the cylinder and $h$ as its height; see Figure~3.
These, together with the length of
$\partial D$, are the geometric invariants that characterize the twisted
cylinder up to isometry. Moreover, $C$ has a parallel foliation by the
nullity lines, i.e. the images of $\{p_0\}\times\R, p_0\in D$, which are
closed if and only if $\theta$ is rational. The interior boundary of $C$
is a flat 2-torus and the limits of the nullity lines induce a parallel
foliation on $\partial_iC$. Observe that $\partial_iC$ also has a
parallel foliation by closed geodesics given by the projection of
$\partial D\times \{s_0\}$, $s_0\in\R$, which will be denoted $\Fol(C)$.

\begin{figure}[!ht]
\centering
\includegraphics[width=0.3\textwidth]{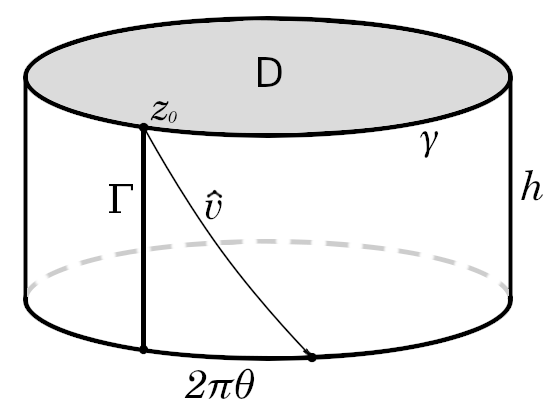}
\caption{\small A twisted cylinder}
\end{figure}

Notice that the action of $\Z$ can be changed
differentiably until $\theta=0$, and hence $C$ is diffeomorphic to a
solid torus $D\times S^1$. According to \cref{twopos}, $M^3$ is thus
either the union of two solid tori glued along their boundary, and hence
diffeomorphic to a lens space, or it is a solid torus whose boundary is
identified via an involution to form a Klein bottle, and therefore
diffeomorphic to a prism manifold.

\begin{rem}\label{orient}
{\rm Let us clarify the role of the orientations in our description of
$C$ in \eqref{g}. Take a twisted cylinder $C$ with nonnegative scalar
curvature, and $D$ a maximal leaf of $\Gamma^\perp$. Orienting $\Gamma$
is then equivalent to orienting $T^\perp D$, which in turn is equivalent
to choosing one of the two generators of $\Z$. On the other hand,
orienting $D$ is equivalent to choosing between the oriented angle
$\theta$ above or $1-\theta$. In particular, these orientations are
unrelated to the metric on $C$, i.e., changing orientations give
isometric cylinders.}
\end{rem}

Next, we show that the \ggm metric on $M^3$ is isotopic to a standard
one. In order to do this, fix once and for all a metric $\met_0$ on the
disc \mbox{$D_0\!=\!\{x\!\in\!\R^2\!:\! |x|\le 1\}$} which is rotationally
symmetric, has positive Gaussian curvature on the interior of $D_0$, and
whose boundary is a closed geodesic of length $1$ along which the
Gaussian curvature vanishes to infinite order. We call the metric on
$M^3$ {\it standard}, if for each twisted cylinder $C=(D\times\R)/\Z$ in
the complement of a core of $M^3$, the metric on $D$ is isometric to
$r^2\met_0$ for some constant $r>0$. Notice that such a metric on $M^3$
is unique up to isometry. For this we first show:

\begin{lem}\label{def}
Let $\met$ be a metric on a disc $D$ with nonnegative Gaussian curvature.
Assume that its boundary is a closed geodesic along which the curvature
vanishes to infinite order, and that the metric is invariant under a
group of isometries $\g$. Then, given a constant $r>0$, there exists a
smooth path of metrics on $D$, $\met_s,\ 1\le s\le 2,$ satisfying
the same assumptions for all $s$, such that $\met_1=\met$ and
$\met_2=r^2\met_0$, where $\met_0$ is the fixed rotationally symmetric
metric on $D_0$.
\end{lem}
\begin{proof}

Let $\met'$ be the standard flat metric on $D_0$.
By the uniformization theorem we can write $\met=f_1^*(e^{2v}\met')$ for
some diffeomorphism $f_1:D\to D_0$ and a smooth function $v$ on $D_0$. The
metric $e^{2v}\met'$ is thus invariant under $C_{f_1}(\g)=\{f_1\circ
g\circ f_1^{-1}: g\in \g\}$ which fixes $f_1(x_0)$, where $x_0\in D$ is
the fixed point of the action of $\g$. Equivalently, $h\in C_{f_1}(\g)$
is a conformal transformation of $(D_0,\met' )$ with conformal factor
$e^{2v-2v\circ h}$. Recall that the conformal transformations of $\met'$
on the interior of $D_0$ can be viewed as the isometry group of the
hyperbolic disc model. Hence there exists a conformal transformation $j$
of $D_0$ with $j(f_1(x_0))=0$ and conformal factor $e^{2\tau}$. We can
thus also write $\met=f^*(e^{2u}\met')$, where $f=j\circ f_1:D\to D_0$
and $u:=(v-\tau)\circ j$. Now the metric $e^{2u}\met'$ is invariant under
$C_f(\g)$, which this time fixes the origin of $D_0$. So $k\in C_f(\g)$
is a conformal transformation of $\met'$ fixing the origin, with
conformal factor $e^{2u-2u\circ k}$. But an isometry of the hyperbolic
disc model, fixing the origin, is also an isometry of $\met'$. Hence
$e^{2u}=e^{2u\circ k}$, i.e. $u$ is invariant under $k$. Altogether,
$C_f(\g)\subset \SO(2)\subset\Iso(D_0,\met')$ and $u$ is
$C_f(\g)$-invariant. Analogously, $r^2\met_0=f_0^*(e^{2u_0}\met')$ with
$f_0\in \Diff(D_0)$ satisfying $f_0(0)=0$ and $u_0$ being
$\SO(2)$-invariant. In particular, $u_0$ is also $C_f(\g)$-invariant.

We now consider the two metrics $e^{2u}\met'$ and $e^{2u_0}\met'$ on
$D_0$. They both have the property that the boundary is a closed geodesic
along which the curvature vanishes to infinite order. An easy computation
shows that the assumption that the boundary is a closed geodesic, up to
parametrization, is equivalent to the condition that the normal
derivatives of $u$ and $u_0$, with respect to a unit normal vector in
$\met'$, is equal to~$1$. Furthermore, since the curvature $G$ of a
metric $e^{2w}\met'$ is given by $Ge^{2w}=-\Delta w$, $G$ vanishes to
infinite order if and only if $\Delta w$ does.
For each $0\leq s\leq 1$, consider the $C_f(\g)$-invariant metric on
$D_0$ given by $\met^s=e^{2(1-s) u_0 +2s u+a(s)}\met'$, where $a(s)$ is
the function that makes the boundary to have length $r$ for all $s$.
Clearly, for each $s$, the boundary is again a closed geodesic up to
parametrization and $G^s$ vanishes at the boundary to infinite order.
Furthermore, since $G^s e^{2(1-s) u_0 +2s u+a(s)}=-(1-s)\Delta u_0
-s\Delta u$ and $\Delta u_0<0,\ \Delta u\leq 0$, the curvature of
$\met^s$ is nonnegative and positive on the interior of $D_0$. Thus
$\met_s=f^*\met^s$ is the desired family of metrics on $D$.
\end{proof}

We can now apply this to deform the metric on $M^3$:

\begin{prop}\label{std}
A \ggm metric with nonnegative scalar curvature is isotopic, through
\ggm metrics with nonnegative scalar curvature, to a standard one.
\end{prop}
\begin{proof}
We define the isotopy separately on each cylinder $C=(D\times\R)/\Z$,
such that the isometry type of the core $H=\partial C$, and the foliation
of $H$ induced by the nullity leaves of $C$, stays fixed. The metric on
$D$ is invariant under the group of isometries $\g=\{g_1\mid (g_1,g_2)\in
\Z\}$ and we apply \lref{def} to obtain a family of metrics $\met_s+
dt^2$ on $D\times \R$, which is invariant under the action of $\Z$. We now
glue the induced metrics on $(D\times\R)/\Z$ to the core $H$ and choose
$r$ such that the arc length parametrization of $\partial C$ and nullity
leaves in $H$ match. Performing this process on each cylinder, we obtain
the desired deformation of the metric on $M^3$.
\end{proof}

We now discuss how $C$ induces a natural marking on its interior boundary
$\partial_iC$. For this, let us first recall some elementary facts about
lattices $\Lambda\subset\R^2$, where we assume that the orientation on
$\R^2$ is fixed.

\begin{definition}
A {\it marking} of the lattice $\Lambda$ is a choice of
an oriented basis $\{v,\hat v\}$ of $\Lambda$, and we say that
the marking is {\it normalized} if
$$
\la v,\hat v\ra/\|v\|^2\in [0,1).
$$
\end{definition}
Notice that for any primitive $v\in \Lambda$, i.e. $tv\not\in\Lambda$ for
$0<t<1$, there exists a unique
oriented normalized marking $\{v,\hat v\}$ of $\Lambda$. Indeed, if
$\{v, w\}$ is some oriented basis of $\Lambda$, then
$\la v,w+n v\ra/\|v\|^2=\la v,w\ra/\|v\|^2+n$ and hence there exists a
unique $n\in\Z$ such that $\{v,\hat v\}$ with $\hat v=w+nv$ is
normalized.

If $T^2$ is an oriented flat torus and $z_0\in T^2$ a base point, then
$T^2=T_{z_0}T^2/\Lambda$ where $\Lambda$ is the lattice given by
$\Lambda = \{w \in T_{z_0}T^2 : exp_{z_0}(w) = z_0 \}$.
A (normalized) marking of $T^2$ is a
(normalized) marking of its lattice $\Lambda$.

\smallskip

Now consider an oriented twisted cylinder $C=(D\times\R)/\Z$ with its
standard metric, where the action of $\Z$ is given by \eqref{g} for some
$\theta$ and $h$. The
totally geodesic flat torus $T^2=\partial_i C$, which inherits an
orientation from $C$, has a natural marking based at $z_0=[(p_0,s_0)]$.
For this, denote by $\gamma:[0,1]\to\partial D$ the simple closed
geodesic  with $\gamma(0)=p_0$ which
follows the orientation of $D=[D\times\{s_0\}]\subset C$.
Then, since $\theta\in [0,1)$, we have that
$$
\M(\gamma):=\{v,\hat v\}, \text{\ \ where \ \ }
v=\gamma'(0) \text{\ \ \ and \ \ }
\hat v=\theta v+h\partial/\partial s,
$$
is a normalized marking of $T^2$ based at $z_0$; see Figure~3. Notice
that the geodesic $\sigma(s)=\exp(s\hat v), 0\le s\le 1$, is simple and
closed with length $\|\hat v\|$. Recall that $\Fol(C)$ denotes the
foliation of $T^2$ by parallel closed geodesics $[\gamma\times \{s\}]$,
$s\in[0,h)$.

\smallskip

It is important for us that the above process can be reversed
for standard metrics:

\begin{prop}\label{deterbo}
Let $T^2$ be a flat oriented torus and $\Fol$ an oriented foliation of
$T^2$ by parallel closed simple geodesics. Then there exists an
oriented twisted cylinder $C_\Fol=(D\times\R)/\Z$ over a standard
oriented disk $D$, unique up to isometry, such that $\partial_iC_\Fol=
T^2$ and $\Fol(C_\Fol)=\Fol$. Moreover, different orientations induce
isometric metrics.
\end{prop}
\proof
Choose $\gamma\in\Fol$, and set $z_0=\gamma(0)$ and $v=\gamma'(0)$. By the above, there exists a unique vector $\hat v$ such that
$\M(\gamma)=\{v,\hat v\}$ is a normalized marking of $T^2$ based at
$z_0$. Set $r=\|v\|$, $\theta = \la v,\hat v\ra/\|v\|^2$
and $h=\|\hat v-\theta v\|$. With respect to the oriented orthonormal
basis $e_1=v/r$, $e_2=(\hat v-\theta v)/h$ of $T_{z_0}T^2$ we have
$$
T^2=\R^2/\Lambda = (\R\oplus\R)/(\Z v\oplus\Z \hat v)
= (S^1_r\times\R)/\Z\hat v,
$$
where $S^1_r$ is the oriented circle of length $r$. Since $v=re_1$ and
$\hat v=\theta v +he_2$, we can also write $T^2= (S^1_r\times\R)/\la g
\ra$ where $g(p,s)=(R_\theta (p),s+h)$. Now we simply attach
$(D_0,r^2\met_0)$ to $S^1_r$ preserving orientations to build
$C=(D_0\times\R)/\la g \ra$. Notice that any two base points of $T^2$ are
taken to each other by an orientation preserving isometry of $C$,
restricted to $\partial C=T^2$. Thus the construction is independent of
the choice of $z_0$ and the choice of $\gamma\in\Fol$.
By \rref{orient}, different choices of orientation induce the same metric
on $C$, and hence $C_\Fol$ is unique up to isometry.
\qed

\smallskip

\begin{rem}
If we do not assume that the metric on $C$ is standard, then the
construction of $C_\Fol$ depends on the choice of base point, and one has
to assume that the metric on $D$ is invariant under $R_\theta$, where
$\theta$ is the angle determined by the marking of $T^2$ induced by
$\Fol$.
\end{rem}

We can now easily classify standard \ggm metrics with two-sided core,
proving case $(a)$ of \tref{ls}.

\begin{theorem}\label{lens}
Let $M^3$ be a compact \ggm of nonnegative scalar curvature with
irreducible universal cover, and assume that its core $T^2$ is two-sided.
Then, $M^3=C_1\sqcup T^2\sqcup C_2$, where $C_i=(D_i\times\R)/\Z$
are twisted cylinders over 2-disks that induce two different foliations
$\Fol_i=\Fol(C_i)$ of $T^2$ by parallel closed geodesics, $i=1,2$.

Conversely, given a flat 2-torus $T^2$ with two different foliations
$\Fol_i$ by parallel closed geodesics, there exists a standard \ggm
$M^3=C_1\sqcup T^2\sqcup C_2$ with irreducible universal cover whose core
is $T^2$ and $C_i=C_{\Fol_i}$. Moreover, this data determines the
standard metric up to isometries, i.e., if $h:T^2\to \hat T^2$ is an
isometry between flat tori, then $\hat M^3=\hat C_1\sqcup \hat T^2\sqcup
\hat C_2$ is isometric to $M^3$, where $\hat C_i=C_{h(\Fol_i)}$.
\end{theorem}
\begin{proof}
We only need to prove uniqueness. The core
of a standard metric is unique since, by the choice of the metric on $D_0$, the set of nonflat points is dense. It is clear then that an
isometry between standard \ggms will send the core to the core, and the
parallel foliations to the parallel foliations. Hence the core and the
parallel foliations are determined by the isometry class of $M^3$.

Conversely, by uniqueness in \pref{deterbo} the standard twisted
cylinders $C_{\Fol_i}$ and $C_{h(\Fol_i)}$ are isometric, which in turn
induces an isometry between $M^3$ and $\hat M^3$. The only ambiguity is
on which side of the torus to attach each of the twisted cylinders,
but this simply gives an orientation reversing isometry fixing the core.
\end{proof}

Now, let us consider the one-sided core case. Here we know that
$M^3=C\sqcup K$ and that $K$ is a nonorientable quotient of the flat
torus $\partial_i C$ and hence a flat Klein bottle. It is easy to see
that, if a flat torus admits an orientation reversing fixed point free
isometric involution~$j$, then $T^2$ has to be isometric to a rectangular
torus $S^1_r\times S^1_s$ on which $j$ acts as in \eqref{pris},
i.e., $j(z,w)=(-z,\bar w)$. Thus,
since the universal cover of $M^3$ is irreducible, $\Fol(C)$ does
not to coincide with one of the two invariant parallel foliations
$\{S^1_r\times \{w\}:w\in S^1_s\}$ and
$\{\{z\}\times S^1_s:z\in S^1_r\}$. We denote the first one by $\Fol(j)$.
\smallskip

As in the proof of \tref{lens}, we conclude:

\begin{theorem}\label{prism}
Let $M^3$ be a compact \ggm of nonnegative scalar curvature with
irreducible universal cover, and assume that its core $K$ is one-sided.
Then $M^3=C\sqcup K$, where $C=(D\times\R)/\Z$ is a
twisted cylinder over a 2-disk with $\partial_i C=T^2$ isometric to a
rectangular torus, and $\partial C=K=T^2/\Z_2$ a flat totally geodesic
Klein bottle.

Conversely, a rectangular flat torus $T^2=S^1_r\times S^1_s$ and a
foliation $\Fol$ of $T^2$ by parallel closed geodesics different from
$S^1_r\times \{p\}$ or $\{p\}\times S^1_s$ define a standard \ggm with
irreducible universal cover $M^3=C_\Fol\sqcup K$ whose core $K$ is
one-sided. Moreover, $T^2$ and $\Fol$ determine $M^3$ up to isometry.
\end{theorem}

We now introduce an isometric invariant of a geometric graph manifold. As
we will see, this invariant determines the diffeomorphism type of the
manifold.

\medskip

For this purpose, we start by defining the slope $\s(\Fol_1,\Fol_2)$
of a foliation $\Fol_2$ by closed simple geodesics of an oriented flat
torus $T^2$ with respect to another such foliation $\Fol_1$.
In order to do this, we first assume that the foliations are oriented.
Fix $z_0\in T^2$, and take $\gamma_i\in\Fol_i$ parametrized over $[0,1]$
such that $\gamma_1(0)=\gamma_2(0)=z_0$. Then $v_i$ is primitive, and as
observed above, there exists a unique $\hat v_i$ such that
$\M(\gamma_i)=\{v_i,\hat v_i\}$ are two normalized markings of $T^2$
based at $z_0$. Since $\SL(2,\Z)$ acts transitively on the set of
oriented bases of a given lattice, there exist coprime integers $p,q$ and
$a,b$ with $bq-ap=1$ such that
\begin{equation}\label{niu}
v_2=qv_1+p\hat v_1,\ \ \hat v_2=av_1+b\hat v_1.
\end{equation}
We also have $p\ne 0$ since $v_1\neq \pm v_2$. Notice that, since
$v_2$ determines $\hat v_2$, the integers $p$~and~$q$ determine $a$ and
$b$.
Observe that $q/p\in\Q$ is independent of the choice of $z_0$
since the foliations are parallel. It does not depend on the orientations
of the foliations either, since $\{-v,-\hat v\}$ is the oriented marking
associated to $-\gamma$. We call
$$\s(\Fol_1,\Fol_2):=q/p$$
the {\it slope}
of $\Fol_2$ with respect to $\Fol_1$. Note though that reversing the
orientation of the torus changes the sign of the slope, since this
corresponds to replacing $\hat v_i$ with $-\hat v_i$. Moreover, since
$v_1=bv_2-p \hat v_2$, we have that $\s(\Fol_2,\Fol_1)=-b/p$.

\medskip

If $M^3=C_1\sqcup T^2\sqcup C_2$ has a two-sided core, a choice of
orientations $\oo=(\oo_M,\oo_T)$ of both $M^3$ and its core $T^2$ orients
the normal bundle of $T^2$. We can thus  choose the order of the two
twisted cylinders $(C_1,C_2)$ by letting $C_1$ be the cylinder containing
the positive direction of the normal bundle. We thus define the {\it
slope} of the lens space as
$$
\s(M^3,\oo)=\s(M^3,(\oo_M,\oo_T)):=\s(\Fol(C_1),\Fol(C_2))\in\Q.
$$
Notice that $\s(M^3,(\oo_M,-\oo_T))=-q/p$ and
$\s(M^3,(-\oo_M,\oo_T))=-b/p$
where $b$ is defined in \eqref{niu}. Since $b=q^{-1} \mod p$, this is
consistent with the fact that $L(p,q)$ and $L(p,q')$ are diffeomorphic if
and only if $q'=\pm q^{\pm1}\!\!\mod p$.

Analogously, if $M^3=C\sqcup K$ has a one-sided core
$K=\partial_iC/\la j \ra$, a choice of an orientation $\oo=\oo_M$
induces an orientation of the torus $\partial_iC$. We call
$\s(M^3,\oo):=\s(\Fol(j),\Fol(C))$ the {\it slope} of the prism manifold,
recalling that $\Fol(j)=\{S^1\times \{w\}: w\in S^1\}$. Here we have
$\s(M^3,-\oo)=-\s(M^3,\oo)$.

Notice that, in either case, the slope of $M^3$ is well
defined even when the \ggm metric is not standard.

\medskip

We now observe:

\begin{prop}
The slope $\s(M^3,\oo)=q/p$ is an oriented isometry invariant of a
\ggm\!\!. Furthermore, the slopes $-q/p$ and $\pm b/p$ are achieved by
changing the orientation on $M^3$ or the core $T^2$. Conversely, any
rational number is the slope of a \ggm\!\!, both on a lens space and on a
prism manifold.
\end{prop}
\begin{proof}
First, assume that $M^3=C_1\sqcup T^2\sqcup C_2$ is a lens space and let
$f\colon M \to M'$ be an orientation preserving isometry. By
\cref{twopos} the core $H$ is unique up to isometry, i.e. there exists a
maximal isometric product $\bar H\times[0,a]\subset M^n$, such that any
$\bar H\times \{s\}$ for $0\le s\le a$ can be regarded as a core, and any
core is of this form. If we choose $H=\bar H\times\{a/2\}$, and similarly
$H'$ for $M'$, then $f$ takes $H$ to $H'$ and by \tref{lens} the isometry
$f|_H$ takes the boundary nullity foliations of $H$ into those of $H'$.
Since we also assume that $f|_H$ is orientation preserving, the slopes of
$M$ and $M'$ are the same. We can argue similarly for a prism manifold,
in which case the core is even unique.

To achieve any slope $q/p$, we can choose the standard basis $e_1,e_2$ of
a product torus $T^2=S^1\times S^1$ and let $v=qe_1+pe_2$. Then there
exists a unique $\hat v$ such that $\{v,\hat v\}$ is a normalized
marking of the torus. This gives rise to two parallel foliations
of $T^2$ with slope $q/p$ and by \tref{lens} they can be realized by a
\ggm metric on a lens space. The same data also gives rise to a prism
manifold by \tref{prism}.
\end{proof}

We are now in position to prove \tref{diffeo} in the introduction, which
states that $\s(M^3,\oo)$ determines the diffeomorphism type of the
manifold.

\medskip

{\it Proof of \tref{diffeo}.}
Recall that the twisted cylinders $C_i$ with invariants $\theta_i,h_i$ as
in \eqref{g} are diffeomorphic to $D_i\times S^1$ by deforming~$\theta_i$
continuously to $0$. For a two-sided core $T^2$, choose
$\gamma_i\in\Fol_i$, and let
$\M(\gamma_i)=\{v_i,\hat v_i\}$ be the normalized markings of $T^2$
defined by $C_i$. Then the natural generators of
$\pi_1(\partial (D_i\times S^1))=\Z\oplus\Z$ are represented by
the simple closed geodesics $\gamma_i$ and $\sigma_i(t)=\exp(t\hat v_i)$,
$0\le t\le 1$, since the marking $\{v_i,\hat v_i\}$ is normalized.
According to the definition of slope,
$v_2=q v_1 +p \hat v_1$ which implies that under the diffeomorphism
from $\partial D_2\times S^1\simeq \partial C_2$ to
$\partial C_1\simeq\partial D_1\times S^1$, the element
$(1,0)\in \pi_1(\partial (D_2\times S^1))$ is taken to
$(q,p)\in \pi_1(\partial (D_1\times S^1))$.
By definition this is the lens space $L(p,q)$; see \sref{prelim}.

To determine the topological type in the one-sided case, we view $M^3$ as
the union of $C$ with the flat twisted cylinder $N^3$ defined in
\eqref{pris}. Then $\partial N^3=T^2$ is a rectangular torus which we
glue to $\partial_i C$. Taking $\e\to0$ (or considering $
T^2\times(0,\e]$ as part of $C$ instead), we obtain $M^3$. We can now use
our second description of prism manifolds in \sref{prelim} and the proof
finishes as in the previous case.
\qed

\smallskip

We finally classify the moduli space of metrics.

\begin{prop}
On a lens space $(L(p,q),\oo)$ the connected components of the moduli
space of \ggm metrics with nonnegative scalar curvature are
pa\-ram\-e\-trized by its slope $q/p\in\Q$, and therefore it has
infinitely many components. On the other hand, on a prism manifold
$P(q,p)$ with $q>1$ the moduli space is connected.
\end{prop}
\begin{proof}
In \pref{std} we saw that we can deform any \ggm metric into one which is
standard. According to \tref{lens}, the standard \ggm metric on a lens
space can equivalently be uniquely defined by the triple
$(T^2,\Fol_1,\Fol_2)$. Thus, we can deform the flat metric on the torus,
carrying along the foliations $\Fol_i$, which induces a deformation of
the original metric by standard metrics. In the proof of \pref{deterbo}
we saw that, after choosing orientations, for $\gamma_i\in\Fol_i$ with
$v_i=\gamma_i'(0)$ we have the normalized markings
$\M(\gamma_i)=\{v_i,\hat v_i\}$ which represents a fundamental domain of
the lattice defined by $T^2$. We can thus deform the flat torus to a unit
square torus such that the first marking is given by $v_1=(1,0),\ \hat
v_1=(0,1)$. Then $v_2=(q,p)=qv_1+p\hat v_1$, which in turn determines
$\hat v_2$, and $q/p$ is the slope of $\Fol_2$ with respect to $\Fol_1$.
Metrics with different slope can clearly not be deformed into each other
since the invariant is a rational number. Since the diffeomorphism type
of the lens space only depends on $\pm q^{\pm 1}$ mod $p$, we obtain
infinitely many components.

For a prism manifold, we similarly deform the metric to be standard and
the rectangular torus into a unit square. But then the absolute value of
its slope already uniquely determines its diffeomorphism type.
\end{proof}

\begin{rems*} $a)$ For a lens space $L(p,q)=\Sph^3/\Z_p$ one can assume
that $p,q>0$, \mbox{$\gcd(p,q)=1$} and $q\leq p$ since the action of
$\Z_p$ is determined by $q$ mod $p$. Then the slopes $q'/p+n$ for
$n\in \N\cup\{0\}$, and $q'=\pm q^{\pm 1}\!\! \mod p$ with $0<q'\leq p$,
parametrize the infinitely many distinct connected components of \ggm
metrics of nonnegative curvature in $L(p,q)$.
Yet, the lens space $L(4p,2p-1)$ has one further component since it is
diffeomorphic to $P(1,p)$. This component is
distinct from the others since the core is one sided.

\smallskip

$b)$
One easily sees that the angle $\alpha$ between the nullity foliations of
a lens space, i.e., the angle between $v_1$ and $v_2$, is given by
$\cos(\alpha)=(q+p\theta_1)r_1/r_2=(b-p\theta_2)r_2/r_1$,
where $r_i=|v_i|$ and $\theta_i$ are the twists of the two cylinders. One
can thus make the nullity leaves orthogonal if and only if $0\leq -q/p<1$
and in
that case $r_2=ph_1,\ h_2=r_1/p$ and $\theta_1=-q/p,\ \theta_2=b/p$. This
determines the metric on the lens space described in the introduction as
a quotient of Figure 1, and is thus the only component containing a
metric with orthogonal nullity leaves.

\smallskip

$c)$ We can explicitly describe the \ggm metrics on $\Sph^3=L(1,1)$ up to
deformation. We assume that the core is a unit square and that the first
foliation is parallel to $(1,0)$, i.e. the first cylinder is a product
cylinder. Then the second marking is given by $v_2=(q,1), \ \hat
v_2=(q-1,1)$. By choosing the orientations appropriately, we can assume
$q\ge 0$. According to the proof of \pref{std}, the marking
$\{v,\hat v\}$ corresponds to a twisted cylinder as in \eqref{g} with
$r=\|v\|$, $\theta = \la v,\hat v\ra/\|v\|^2$ and
$h=\|\hat v-\theta v\|$. Thus in our case the second cylinder is given by
$r=1/h=\sqrt{1+q^2}$, and $\theta=(1+q^2-q)(1+q^2)$. The slope is $q$,
and the standard metric in Figure 1 corresponds to $q=0$.

\end{rems*}

\end{document}